\documentclass[12pt]{amsart}
\usepackage{amssymb,amsmath,amsfonts,amsthm,latexsym}
\usepackage{verbatim}

\numberwithin{equation}{section}
\numberwithin{table}{section}

\theoremstyle{plain}
\newtheorem{theorem}{Theorem}[section]
\newtheorem{lemma}[theorem]{Lemma}

\theoremstyle{remark}
\newtheorem*{remark}{Remark}

\newcommand{\zl}[1]{\mathbb{Z}^{\mathcal{L}_{{#1}}}}
\newcommand{\fl}[1]{\mathbb{F}_{p}^{\mathcal{L}_{{#1}}}}
\DeclareMathOperator{\im}{Im}

\title[Elementary Divisors]{The Elementary Divisors of the Incidence Matrices of Skew Lines in $PG(3,p)$}
\author{Joshua Ducey} 
\author{Peter Sin} 
\address{Department of Mathematics, University of Florida, Gainesville, FL 32611--8105, USA.}
\email{jducey21@ufl.edu}
\email{sin@ufl.edu}


\begin{document}
\begin{abstract}
The elementary divisors of the incidence matrices of lines in $PG(3,p)$ are computed, where two lines are incident if and only if they are skew.  
\end{abstract}
\maketitle
\section{Introduction}
Let $V$ be a 4-dimensional vector space over the finite field $\mathbb{F}_{p}$ of $p$ elements, where $p$ is a prime.  We declare two 2-dimensional subspaces $R$ and $S$ to be incident if and only if $R \cap S = \{0\}$.  Ordering the 2-dimensional subspaces in some arbitrary but fixed manner, we can form the incidence matrix $A$ of this relation.  In this paper we compute the elementary divisors of $A$ viewed as an integer matrix.  In order to introduce some useful notation, we will view this setup as a special case of a more general situation.  For brevity, an $r$-dimensional subspace will be called an $r$-subspace in what follows.

More generally, let $V$ be an (n+1)-dimensional vector space over the finite field $\mathbb{F}_{q}$, where $q = p^{t}$ is a prime power.  Let $\mathcal{L}_{r}$ denote the set of $r$-subspaces of $V$.  Thus $\mathcal{L}_{1}$ denotes the points, $\mathcal{L}_{2}$ denotes the lines, etc. in $\mathbb{P}(V)$.  An $r$-subspace $R \in \mathcal{L}_{r}$ and an $s$-subspcace $S \in \mathcal{L}_{s}$ are incident if and only if $R \cap S = \{0\}$.  The incidence matrix with rows indexed by the $r$-subspaces and columns indexed by the $s$-subspaces is denoted $A_{r,s}$.  We can and do view $A_{r,s}$ as the matrix of a homomorphism of free $\mathbb{Z}$-modules 
\[
\eta_{r,s}: \zl{r} \to \zl{s}
\]
that sends an $r$-subspace to the (formal) sum of all $s$-subspaces incident with it.  Then computing the elementary divisors of $A_{r,s}$ is the same thing as finding a cyclic decomposition of the cokernel of $\eta_{r,s}$.

These matrices $A_{r,s}$ have been studied by several authors.  Their $p$-ranks can be deduced from the work~\cite{sin:2004}, but their elementary divisors have been computed only in the case that either $r$ or $s$ is equal to one~\cite{sin:2000, chandler:sin:xiang:2006}.

We return now and for the remainder of the paper to the case where $V$ is 4-dimensional over the prime field $\mathbb{F}_{p}$, and $A = A_{2,2}$.
\begin{theorem}  \label{thm:main}
The elementary divisors of the incidence matrix $A$ are all $p$-powers, and are as given in the table below.
\begin{center}
\begin{tabular}{|c|c|} 
\hline
Elem. Div. & Multiplicity       \\ \hline
\ $1$ & $p(2p^{2}+1)/3$           \\
\ $p$ & $p(3p^3-2p^2+3p-1)/3$   \\
\ $p^2$ & $p(p+1)(p+2)/3$       \\
\ $p^3$ & $p(2p^{2}+1)/3$         \\
\ $p^4$ & $1$                   \\
\hline
\end{tabular}
\end{center}
\end{theorem}

\begin{remark}
The fact that the elementary divisors are $p$-powers follows from more general work of Andries Brouwer~\cite{brouwer}.  Indeed, the problem that is the subject of this paper was suggested by Brouwer to the second author at the recent workshop on invariants of incidence matrices held at the Banff International Research Station.
\end{remark}
The rest of the paper will be organized as follows.  In the next section, the determinant of $A$ will be computed.  It will follow at once from this computation that the cokernel of $\eta_{2,2}$ is a finite $p$-group.  We will then introduce a certain $p$-filtration of a $GL(4,p)$-module that allows for a convenient reformulation of our problem.  The proof of the theorem relies on the submodule structure of this $GL(4,p)$-module.
\section{The Determinant of A}
\label{sec:det}	
Let $x \in \mathcal{L}_{2}$.  Then we have 
\[
\eta_{2,2}(\eta_{2,2}(x)) = \sum_{y \in \mathcal{L}_{2}}a_{x,y}y
\]
where
\begin{align*}
a_{x,y} &= \lvert\{z \in \mathcal{L}_{2} \, \vert \, z \cap x = \{0\} \hbox{ and } z \cap y = \{0\}\}\rvert \\
&= \begin{cases}
   p^{4}, & \hbox{ if } y = x \\
   p^{4}-p^{3}-p^{2}+p, & \hbox{ if } y \neq x \hbox{  and } y \cap x = \{0\} \\
   p^{4} - p^{3}, & \hbox{ if } y \neq x \hbox{ and } y \cap x \neq \{0\}
   \end{cases}
\end{align*}
as follows by an elementary counting argument.  Keeping in mind that $x$ is incident with $y$ if and only if $x \cap y = \{0\}$, from the above we immediately get the relation
\[
A^{2} = p^{4}I + (p^{4} - p^{3} - p^{2} + p)A + (p^{4} - p^{3})(J - A - I),
\]
which reduces to
\begin{equation}
\label{eq:matrix}
A^{2} = p^{3}I - (p^{2} - p)A + (p^{4} - p^{3})J,
\end{equation}
where $I$ and $J$ denote the $|\mathcal{L}_{2}| \times |\mathcal{L}_{2}|$ identity matrix and all-one matrix, respectively.

Since all matrices involved in~\eqref{eq:matrix} are symmetric and commute with each other, we can simultaneously diagonalize to obtain a system of $|\mathcal{L}_{2}|$ equations:
\begin{align*}
\lambda_{1}^{2} &= p^{3} - (p^{2} - p)\lambda_{1} + (p^{4} - p^{3})(p^{2} + 1)(p^{2} + p + 1), \\
\lambda_{i}^{2} &= p^{3} - (p^{2} - p)\lambda_{i} \, , \qquad 2 \le i \le |\mathcal{L}_{2}|
\end{align*}
where $\lambda_{1}$, $\lambda_{2}$, . . . are the eigenvalues of $A$ counted with multiplicity.  Solving the quadratic equations, we see that $\lambda_{1}$ is either $p^{4}$ or $p - p^{2} - p^{4}$, and that for each $i \ge 2$ we have that $\lambda_{i}$ is either $p$ or $-p^{2}$.  Since the all-one vector is an eigenvector for $A$ with eigenvalue $p^{4}$, we must have that $\lambda_{1} = p^{4}$.  Using the fact that the trace of $A$ is zero, one can easily solve for the multiplicities of the remaining eigenvalues of $A$.  The eigenvalue $p$ has multiplicity $p^{4} + p^{2}$ and the eigenvalue $-p^{2}$ has multiplicity $p^{3} + p^{2} + p$.  This yields:
\begin{equation}
\label{eq:det}
\det(A) = p^{4} \cdot p^{(p^{4} + p^{2})} \cdot (-p^{2})^{(p^{3} + p^{2} + p)} = (-1)^{p} p^{(p^{4} + 2p^{3} + 3p^{2} + 2p + 4)}.
\end{equation}
\section{Some Modules for $GL(4,p)$}
\label{sec:action}
Set $G = GL(4,p)$.  If we fix a basis of $V$ then $G$ acts transitively on the sets $\mathcal{L}_{r}$ $(r=1,2,3)$ and the incidence map $\eta_{2,2}: \zl{2} \to \zl{2}$ is clearly a homomorphism of $\mathbb{Z}G$-permutation modules.  Reduction mod $p$ induces an $\mathbb{F}_{p}G$-permutation module homomorphism, denoted $\overline{\eta_{2,2}}: \fl{2} \to \fl{2}.$  Define a sequence of $\mathbb{Z}G$-submodules $\{M_{i}\}_{i\ge0}$ of $\zl{2}$ as follows.  Put $M_{0} = \zl{2}$, and for $i\ge1$ put
\[
M_{i} = \{m \in \zl{2} \, \vert \, \eta_{2,2}(m) \in p^{i}\zl{2}\}.
\]
Thus we have a descending filtration $\zl{2} = M_{0} \supseteq M_{1} \supseteq M_{2} \supseteq \cdot \cdot \cdot$ of $\zl{2}$.  For a submodule $N$ of $\zl{2}$, denote by $\overline{N}$ its image in $\fl{2}$; in other words $\overline{N} = (N + p\zl{2})/p\zl{2}$.  We have an induced $p$-filtration
\[
\fl{2} = \overline{M_{0}} \supseteq \overline{M_{1}} \supseteq \overline{M_{2}} \supseteq \cdot \cdot \cdot
\]
of $\mathbb{F}_{p}G$-submodules.  With a little thought, one sees that for each $i \ge 0$ the multiplicity of $p^{i}$ as an elementary divisor of $A$ is precisely $\dim_{\mathbb{F}_{p}}(\overline{M_{i}}/\overline{M_{i+1}})$.  This reformulation of our problem allows us to focus our attention on the $\mathbb{F}_{p}G$-submodule structure of $\fl{2}$, which we now do.

For $r = 1,2,3$ define
\[
Y_{r} = \left\{\sum_{x \in \mathcal{L}_{r}}a_{x}x \in \fl{r} \, \big\vert \, \sum_{x \in \mathcal{L}_{r}}a_{x} = 0\right\}.
\]
We use $\mathbf{1}$ to denote the element $\sum_{x \in \mathcal{L}_{r}}x \in \zl{r}$, and the same symbol also for its image in $\fl{r}$ ($r$ will always be clear from the context).  Since $|\mathcal{L}_{r}| \equiv 1 \pmod p$ we have the decomposition
\begin{equation}
\label{eq:decomp}
\fl{r} = \mathbb{F}_{p}\mathbf{1} \oplus Y_{r}.
\end{equation}
Thus the essential task is to understand $Y_{2}$.  Crucial for our analysis is the $\mathbb{F}_{p}G$-submodule structure of $Y_{1}$ and $Y_{3}$, which is well-known and given in the theorem below.

For $i=1,2,3$ let $S_{i}$ denote the degree $i(p-1)$ component of the graded algebra $S^{*}(V)/(V^{p})$, the quotient of the symmetric algebra on $V$ by the ideal generated by $p$-powers.  $S_{i}$ is a simple $\mathbb{F}_{p}G$-module with dimension equal to the coefficient of $x^{i(p-1)}$ in the expansion of $(1 + x + \cdot \cdot \cdot + x^{p-1})^{4}$.  Explicitly, we have
\begin{equation}
\label{eq:omega1}
\dim_{\mathbb{F}_{p}}S_{1} = \dim_{\mathbb{F}_{p}}S_{3} = \frac{p(p+1)(p+2)}{6}
\end{equation}
and
\begin{equation}
\label{eq:omega2}
\dim_{\mathbb{F}_{p}}S_{2} = \frac{p(2p^{2}+1)}{3}.
\end{equation}
\begin{theorem}[Cf.~\cite{sin:2000}, Theorem 2 and the comment following it.]
\label{thm:aux}
\hfil
\begin{enumerate}
\item \label{item:auxY1} $Y_{1}$ is uniserial with composition series
\[
Y_{1} = W_{1} \supseteq W_{2} \supseteq W_{3} \supseteq W_{4} = \{0\}
\]
and simple quotients
\[
W_{i}/W_{i+1} \cong S_{i}, \qquad \hbox{for } i=1,2,3.
\]
\item \label{item:auxY3} $Y_{3}$ is uniserial with composition series
\[
Y_{3} = U_{3} \supseteq U_{2} \supseteq U_{1} \supseteq U_{0} = \{0\}
\]
and simple quotients
\[
U_{i}/U_{i-1} \cong S_{i}, \qquad \hbox{for } i=1,2,3.
\]
\end{enumerate}
\end{theorem}

The following two lemmas provide all of the information about the submodule structure of $Y_{2}$ that we will need in order to prove Theorem~\ref{thm:main} in the next section. Before we can state the lemmas we must first define two important maps. Let $\phi: \zl{1} \to \zl{2}$ be the $\mathbb{Z}G$-module homomorphism defined by
\[
\phi(x) = \sum_{\substack{y \in \mathcal{L}_{2} \\ x \cap y = \{0\}}} y, \qquad \hbox{for all } x \in \mathcal{L}_{1}.
\]
(So $\phi$ is just $\eta_{1,2}$.) The other map we call $\psi: \zl{3} \to \zl{2}$, and it is the $\mathbb{Z}G$-module homomorphism given by
\[
\psi(x) = \sum_{\substack{y \in \mathcal{L}_{2} \\ x \cap y \neq y}} y, \qquad \hbox{for all } x \in \mathcal{L}_{3}.
\]
Denote their reductions mod $p$ by $\overline{\phi}$ and $\overline{\psi}$, respectively. Notice how these maps respect the decomposition~\eqref{eq:decomp}.  In particular,
\[
\overline{\eta_{2,2}}(\mathbf{1}) = \overline{\phi}(\mathbf{1}) = \overline{\psi}(\mathbf{1}) = 0.
\]
Hence
\begin{align*}
\im\overline{\eta_{2,2}} &= \overline{\eta_{2,2}}(Y_{2}), \\
\im\overline{\phi} &= \overline{\phi}(Y_{1}), \\
\im\overline{\psi} &= \overline{\psi}(Y_{3}),
\end{align*}
and these are all submodules of $Y_{2}$.  This observation will be used frequently.
\begin{lemma}
\label{lem:incl}
\hfil
\begin{enumerate}
\item \label{item:incl1} $\ker\overline{\eta_{2,2}} \subseteq \overline{M_{1}}$.
\item \label{item:incl2} $\mathbb{F}_{p}\mathbf{1} \oplus (\im\overline{\phi} + \im\overline{\psi}) \subseteq \overline{M_{2}}$.
\item \label{item:incl3} $\mathbb{F}_{p}\mathbf{1} \oplus \im\overline{\eta_{2,2}} \subseteq \overline{M_{3}}$.
\item \label{item:incl4} $\mathbb{F}_{p}\mathbf{1} \subseteq \overline{M_{4}}$.
\end{enumerate}
\end{lemma}
\begin{proof}
It follows immediately from the definition of $\overline{M_{1}}$ that, in fact, $\ker\overline{\eta_{2,2}} = \overline{M_{1}}$.  So~\eqref{item:incl1} holds.  Since $\eta_{2,2}(\mathbf{1}) = p^{4}\mathbf{1}$, \eqref{item:incl4} follows.

We can rewrite equation~\eqref{eq:matrix} as
\[
(A + (p^{2} - p)I)A = p^{3}I + (p^{4} - p^{3})J.
\]
Then if we define $\gamma : \zl{2} \to \zl{2}$ by $\gamma(x) = \eta_{2,2}(x) + (p^{2} - p)x$, for all $x \in \mathcal{L}_{2}$, it follows from the above matrix equation that $\eta_{2,2}(\im\gamma) \subseteq p^{3}\zl{2}$.  Thus $\im\gamma \subseteq M_{3}$ and so $\im\overline{\eta_{2,2}} = \overline{\im\eta_{2,2}} = \overline{\im\gamma} \subseteq \overline{M_{3}}$.  This establishes~\eqref{item:incl3}.

Let $x \in \mathcal{L}_{1}$.  Then we have
\[
\eta_{2,2}(\phi(x)) = \sum_{y \in \mathcal{L}_{2}}b_{x,y}y
\]
where 
\begin{align*}
b_{x,y} &= |\{z \in \mathcal{L}_{2} \, \vert \, z \cap x = \{0\} \hbox{ and } z \cap y = \{0\}\}| \\
&= \begin{cases}
   p^{4}, & \hbox{ if } x \cap y \neq \{0\} \\
   p^{4}-p^{2}, & \hbox{ if } x \cap y = \{0\}
   \end{cases}
\end{align*}
as is easily checked.  Thus $\eta_{2,2}(\phi(x)) \in p^{2}\zl{2}$ and so $\im\overline{\phi} \subseteq \overline{M_{2}}$.  By a dual argument, $\im\overline{\psi} \subseteq \overline{M_{2}}$ establishing~\eqref{item:incl2} and proving the lemma.
\end{proof}
\begin{lemma}
\label{lem:img}
\hfil
\begin{enumerate}
\item \label{item:img1} $Y_{2}$ has a simple head and socle, and $\mathrm{head}(Y_{2}) \cong \mathrm{soc}(Y_{2}) \cong S_{2}.$
\item \label{item:img2} $\im\overline{\eta_{2,2}} = \mathrm{soc}(Y_{2})$.
\item \label{item:img3} $(\im\overline{\phi} + \im\overline{\psi})/\mathrm{soc}(Y_{2}) \cong S_{1} \oplus S_{3}$.
\end{enumerate}
\end{lemma}
\begin{proof}
Since the stabilizer or a $2$-subspace has $3$ orbits on $\mathcal{L}_{2}$, we have
\begin{equation}
\label{eq:dim3}
\dim_{\mathbb{F}_{p}}\mathrm{End}_{\mathbb{F}_{p}G}(\fl{2}) = 3.
\end{equation}
If $Y_{2}$ contained the trivial module $\mathbb{F}_{p}$ then it would follow from the decomposition~\eqref{eq:decomp} that $\dim_{\mathbb{F}_{p}}\mathrm{End}_{\mathbb{F}_{p}G}(\fl{2}) \geq 4$, a contradiction.  Thus the multiplicity of $\mathbb{F}_{p}$ in $\mathrm{soc}(Y_{2})$ is $0$.

Let $L$ be a simple, non-trivial $\mathbb{F}_{p}G$-module.  By Frobenius reciprocity we have
\begin{align*}
\mathrm{Hom}_{\mathbb{F}_{p}G}(L, Y_{2}) &\cong \mathrm{Hom}_{\mathbb{F}_{p}G}(L, \fl{2}) \\
&\cong \mathrm{Hom}_{\mathbb{F}_{p}P}(L, \mathbb{F}_{p}),
\end{align*}
where $P$ is the stabilizer of a $2$-subspace.  Thus if $L$ is a simple submodule of $Y_{2}$, then by dualizing we see that the multiplicity of $L$ in $\mathrm{soc}(Y_{2})$ is $\dim_{\mathbb{F}_{p}}(L^{*})^{P}$.  Therefore $P$ fixes a nonzero vector of $L^{*}$, and it follows from the general theory of modular representations of finite groups of Lie type~\cite{curtis:1970,steinberg:1963} that $L^{*}$ is determined up to isomorphism by this property.  Since $P$ has a fixed point on $S_{2}$, we thus have $S_{2} \cong L^{*}$.  Furthermore the fixed vector is unique up to scalars, so $\mathrm{soc}(Y_{2})$ is simple.  Since $S_{2}$ and $Y_{2}$ are both self-dual, ~\eqref{item:img1} follows.

Let $\pi_{\mathbf{1}}$ and $\pi_{Y_{2}}$ denote the projections of $\fl{2}$ onto $\mathbb{F}_{p}\mathbf{1}$ and $Y_{2}$, respectively.  By~\eqref{item:img1} we can find an isomorphism $\alpha : \mathrm{head}(Y_{2}) \to \mathrm{soc}(Y_{2})$ of $\mathbb{F}_{p}G$-modules.  Then define $\mu \in \mathrm{End}_{\mathbb{F}_{p}G}(\fl{2})$ by $\mu = \alpha \circ q \circ  \pi_{Y_{2}}$, where $q$ denotes taking the quotient onto $\mathrm{head}(Y_{2})$.  Clearly $\{\pi_{\mathbf{1}}, \pi_{Y_{2}}, \mu\}$ are independent and thus give a basis of $\mathrm{End}_{\mathbb{F}_{p}G}(\fl{2})$ by equation~\eqref{eq:dim3}.

Thus we can write $\overline{\eta_{2,2}} = a\pi_{\mathbf{1}} + b\pi_{Y_{2}} + c\mu$ for some $a,b,c \in \mathbb{F}_{p}$. Since $\im\overline{\eta_{2,2}} \subseteq Y_{2}$ we see that $a=0$.  Also, the restriction of $\overline{\eta_{2,2}}$ to $Y_{2}$ has a nontrivial kernel (it contains, for example, $\overline{\eta_{2,2}}(x-y)$, where $x,y \in \mathcal{L}_{2}$).  Thus $\im\overline{\eta_{2,2}}$ is a proper submodule of $Y_{2}$ and so $b=0$. So $\overline{\eta_{2,2}}$ is a nonzero multiple of $\mu$, hence $\im\overline{\eta_{2,2}} = \im\mu = \mathrm{soc}(Y_{2})$. This proves~\eqref{item:img2}.

Since $\overline{\phi} \neq 0$ and $\im\overline{\phi} \subseteq Y_{2}$, we have by~\eqref{item:img1} that $\mathrm{soc}(Y_{2}) \subseteq \im\overline{\phi}$. The fact that $\overline{\phi}(\mathbf{1}) = 0$ implies that $\im\overline{\phi}$ is isomorphic to a quotient of $Y_{1}$. Then it follows from Theorem~\ref{thm:aux} and \eqref{item:img1} that $\im\overline{\phi} \cong Y_{1}/W_{3}$. Thus $\im\overline{\phi}$ has $\mathrm{soc}(Y_{2})$ as its unique proper submodule and
\[
\im\overline{\phi}/ \mathrm{soc}(Y_{2}) \cong S_{1}.
\]
Similarly, Theorem~\ref{thm:aux} and \eqref{item:img1} imply that $\im\overline{\psi} \cong Y_{3}/U_{1}$. Thus $\im\overline{\psi}$ has $\mathrm{soc}(Y_{2})$ as its unique proper submodule and
\[
\im\overline{\psi}/ \mathrm{soc}(Y_{2}) \cong S_{3}.
\]
Since $\im\overline{\phi} \cap \im\overline{\psi} = \mathrm{soc}(Y_{2})$, ~\eqref{item:img3} follows.
\end{proof}
\section{Proof of Theorem~\ref{thm:main}}
\label{sec:proof}
For $i\ge0$, let $f_{i}$ denote the multiplicity of $p^{i}$ as an elementary divisor of $A$.  Thus $f_{i} = \dim_{\mathbb{F}_{p}}(\overline{M_{i}}/ \overline{M_{i+1}})$.  For $i = 0,1,2,3,4$ let $e_{i}$ denote the claimed multiplicity of $p^{i}$ as listed in the statement of Theorem~\ref{thm:main}.  Then the proof will be completed if we show that $e_{i} = f_{i}$ for $i=0,1,2,3,4$ and $f_{i} = 0$ for $i \ge 5$.

Since $f_{0}$ is precisely the $p$-rank of $A$, we have by Lemma~\ref{lem:img} that
\begin{align*}
f_{0} &= \dim_{\mathbb{F}_{p}}\im\overline{\eta_{2,2}} \\
&= \dim_{\mathbb{F}_{p}}\mathrm{soc}(Y_{2}) \\
&= \dim_{\mathbb{F}_{p}}S_{2} \\
&= \frac{p(2p^{2}+1)}{3} \\
&= e_{0}.
\end{align*}

We proceed by applying Lemma~\ref{lem:img} and our earlier calculations~\eqref{eq:omega1} and~\eqref{eq:omega2} to compute the dimensions of the modules that occur on the left hand side of Lemma~\ref{lem:incl}.  We clearly have
\[
\dim_{\mathbb{F}_{p}}\mathbb{F}_{p}\mathbf{1} = 1 = e_{4}.
\]
Also by Lemma~\ref{lem:img} we compute
\begin{align*}
\dim_{\mathbb{F}_{p}}(\mathbb{F}_{p}\mathbf{1} \oplus \im\overline{\eta_{2,2}}) &= 1+\dim_{\mathbb{F}_{p}}\mathrm{soc}(Y_{2}) \\
&= 1+\dim_{\mathbb{F}_{p}}S_{2} \\
&= 1 + \frac{p(2p^{2}+1)}{3} \\
&= e_{3} + e_{4}
\end{align*}
and
\begin{align*}
\dim_{\mathbb{F}_{p}}(\mathbb{F}_{p}\mathbf{1} \oplus (\im\overline{\phi} + \im\overline{\psi})) &= 1+\dim_{\mathbb{F}_p}\mathrm{soc}(Y_{2}) + \dim_{\mathbb{F}_{p}}((\im\overline{\phi} + \im\overline{\psi})/ \mathrm{soc}(Y_{2})) \\
&= 1+\dim_{\mathbb{F}_{p}}S_{2} + \dim_{\mathbb{F}_{p}}(S_{1} \oplus S_{3}) \\
&= 1 + \frac{p(2p^{2}+1)}{3} + 2\cdot\frac{p(p+1)(p+2)}{6} \\
&= e_{2} + e_{3} + e_{4}.
\end{align*}
A final application of Lemma~\ref{lem:img} and direct computation gives
\begin{align*}
\dim_{\mathbb{F}_{p}}\ker\overline{\eta_{2,2}} &= \dim_{\mathbb{F}_{p}}\fl{2} - \dim_{\mathbb{F}_{p}}\im\overline{\eta_{2,2}} \\
&= \dim_{\mathbb{F}_{p}}\fl{2} - \dim_{\mathbb{F}_{p}}\mathrm{soc}(Y_{2}) \\
&= \dim_{\mathbb{F}_{p}}\fl{2} - \dim_{\mathbb{F}_{p}}S_{2} \\
&= (p^{2}+p+1)(p^{2}+1) - \frac{p(2p^{2}+1)}{3} \\
&= (3p^{4} + p^{3} +6p^{2} + 2p +3)/3 \\
&= e_{1} + e_{2} + e_{3} + e_{4}.
\end{align*}

These computations together with the inclusions stated in Lemma~\ref{lem:incl} give us the inequalities:
\begin{align*}
e_{1} + e_{2} + e_{3} + e_{4} &\leq f_{1} + f_{2} + f_{3} + f_{4} + f_{5} + \cdots , \\
e_{2} + e_{3} + e_{4} &\leq f_{2} + f_{3} + f_{4} + f_{5} + \cdots , \\
e_{3} + e_{4} &\leq f_{3} + f_{4} + f_{5} + \cdots , \\
e_{4} &\leq f_{4} + f_{5} + \cdots .
\end{align*}
Summing the left and right hand sides of this system gives us
\[
e_{1} + 2e_{2} + 3e_{3} + 4e_{4} \leq f_{1} + 2f_{2} + 3f_{3} + 4f_{4} + 5f_{5} + \cdots
\]
and therefore
\[
p^{e_{1} + 2e_{2} + 3e_{3} + 4e_{4}} \leq p^{f_{1} + 2f_{2} + 3f_{3} + 4f_{4} + 5f_{5} + \cdots}.
\]
The right hand side of this inequality is exactly $|\det(A)|$.  By direct calculation and equation~\eqref{eq:det} we find that the left hand side is also $|\det(A)|$. It follows immediately that we have the \textit{equalities}
\begin{align*}
e_{1} + e_{2} + e_{3} + e_{4} &= f_{1} + f_{2} + f_{3} + f_{4} + f_{5} + \cdots , \\
e_{2} + e_{3} + e_{4} &= f_{2} + f_{3} + f_{4} + f_{5} + \cdots , \\
e_{3} + e_{4} &= f_{3} + f_{4} + f_{5} + \cdots , \\
e_{4} &= f_{4} + f_{5} + \cdots ,
\end{align*}
and so
\begin{align*}
e_{1} &= f_{1} , \\
e_{2} &= f_{2} , \\
e_{3} &= f_{3}.
\end{align*}
Finally, since $e_{4} = 1$ it follows easily that $f_{4} = 1$ and $f_{i} = 0$ for $i \ge 5$.
\qed
\bibliographystyle{plain}
\bibliography{reference}
\end{document}